\documentclass{article}
\usepackage{amsfonts}
\usepackage{amssymb, amsmath, amsthm}
\usepackage{algorithm}
\usepackage{algorithmic}
\usepackage[letterpaper,body={13.5cm,22cm}, mag=1000]{geometry}
\usepackage{amssymb}
\usepackage{secdot}
\usepackage{setspace}
\usepackage{titlesec}
\titleformat{\section}[runin]{\bfseries\filcenter}{\thesection}{1em}{}

\renewcommand{\thesection}{\arabic{section}}
\title{\large \bf Finite groups with specific number of cyclic subgroups }
{

{
\author{\small \bf Hemant Kalra  \\
\small \em School of Mathematics \\
\small \em Thapar University, Patiala - 147 004,
India\\
\small E-mail:happykalra26@gmail.com\\
}}
\date{}

\newtheorem{thm}{Theorem}[section]
\newtheorem{lm}[thm]{Lemma}
\newtheorem{pp}[thm]{Proposition}

\newtheorem{cor}[thm]{Corollary}
\newtheorem{rem}[thm]{Remark}

\begin{document}

\maketitle
\begin{abstract}
In this note, we classify all finite groups having exactly $6$, $7$ or $8$ cyclic subgroups. This gives a partial answer to the open problem posed by T\u arn\u auceanu (Amer. Math. Monthly, {\bf 122} (2015), 275-276). As a consequence of our results, we also obtain an important result concerning  with the converse of Lagrange's theorem.

\end{abstract}
\vspace{2ex}


\vspace{2ex}


\section{Introduction} Let $G$ be a finite group and let $c(G)$ denote the set of all cyclic subgroups of $G$. In this note, we work on a basic problem of  group theory. It is always interesting and worthwhile to classify finite groups with specific properties.  Recently, T\u arn\u auceanu ${\cite{tar}}$ proved that a finite group $G$ has $|G|-1$ cyclic subgroups if and only if $G$ is one of the $C_3, C_4, S_3$ or $D_8$. In the same paper, he posed an open problem:\\

\noindent {\bf{ Open problem}}. Describe the finite groups $G$ satisfying $|c(G)|=|G|-r$, where $2\le r\le |G|-1$.\\

Motivated by his work, Zhou \cite{zho} found all finite groups $G$ with $|c(G)|$ equal to 3, 4 or 5. In this note, we  give a partial answer to the above posed problem and classify all finite groups $G$ satisfying $|c(G)|=|G|-r$, where $r=|G|-6, |G|-7$ or $r=|G|-8$. Equivalently, we classify all finite groups $G$ having exactly $6$, $7$ or $8$ cyclic subgroups. As a consequence, we also give a very short and easy proof  of the main result of Zhou \cite{zho}. In Section 3, we collect some basic results. We compute number of cyclic subgroups of $G$ whenever order of $G$ is $pq$ or $p^2q$, where $p$ and $q$ are primes. As a consequence, we also prove that $A_4,$ (the group of all even permutations on four symbols) is the only  group of the order $p^2q$, where $p<q$, for which the converse of Lagrange's theorem is not true. We, also compute the order of $c(G)$, when $G$ is a finite $p$-group of order less or equal to $p^4$. Using these results as building blocks, we derive our main results in Section 4 which classifies all finite groups $G$ with $|c(G)|=6$, $7$ or $8$. 

\section{Notations and Preliminary Lemmas} 
For any positive integer $n$, let  $d(n)$ denote the number of positive divisors $n$ and $\phi(n)$ denote the number of positive integers less than and relative prime to $n$ respectively;  by $C_n$ we mean  a cyclic group of order $n$. For a prime $p$, $n_p$ denote the number of Sylow $p$-subgroups of the group $G$. Observe that, if $G$ is a finite elementary abelian group of order $p^n$, then $G$ has $p^{n-1}+p^{n-2}+\ldots +p+2$ cyclic subgroups.  The following result is due to  Richards \cite[Theorem]{ric}. We use this result quite frequently in the proofs of Section 4 without any further reference.

\begin{lm}
Let $G$ be a  group of order $n$. The number of cyclic subgroups of $G$ is greater than or equal to $d(n)$. Furthermore, the number of cyclic subgroups of $G$ is $d(n)$ if and only if $G$ is cyclic. 
\end{lm}

\begin{lm}
Let $G$ be a finite group and let $n$ divides the order of $G$. Then the number of cyclic subgroups of order $n$ is equal to $\frac{\mbox{number of elements of order n}}{\phi(n)}$.
\end{lm}

\section{Groups of Small Order}
In this section, we compute number of cyclic subgroups of $G$, when order of $G$ is $pq$ or $p^2q$, where $p$ and $q$ are distinct primes. We also show that there is a close relation in computing $|c(G)|$ and the converse of Lagrange's theorem.  

\begin{lm}
Let $G$ be a finite non-abelian group of order $pq$, where $p$ and $q$ are distinct primes and $p<q$. Then $|c(G)|=q+2$.
\end{lm}

\begin{proof} 
It follows from Sylow's theorem that $G$ has a unique Sylow $q$-subgroup. Since $G$ is non-abelian, the number of Sylow $p$-subgroups of $G$ are $q$. Also, the group $G$ has an  identity subgroup. Thus $G$ has $q+2$ cyclic subgroups.
 \end{proof}

\begin{pp}
Let $G$ be a finite   group of order $pq^2$, where $p$ and $q$ are distinct primes and $p<q$. Then $|c(G)|=6, 2q+4, q^2+3, q^2+q+2, 2q+3,$ or $3q+2$.
\end{pp}

\begin{proof} First suppose that $G$ is abelian. If $G$ is cyclic, then $|c(G)|=6$. Assume that $G$ is non-cyclic abelian group. It follows from the Sylow's theorem that the group $G$ has a unique Sylow $p$-subgroup and a unique Sylow $q$-subgroup
 respectively and hence $(p-1)(q^2-1)$ elements of order $pq$. Thus $G$ has $2q+4$ cyclic subgroups. Next, we assume that $G$ is non-abelian. It follows from Sylow's theorem that $n_q=1$ and  $n_p$ is either $q$ or $q^2$. Observe that if $n_p=q^2$, then $G$ has $q^2(p-1)$ elements of order $p$ and hence  $G$ has no element of order $pq$. First suppose that the unique Sylow $q$-subgroup is cyclic and  $n_p=q^2$. Thus   $|c(G)|=q^2+3$.  Now, suppose that the unique Sylow $q$-subgroup of $G$ is elementary abelian. Thus $G$ has $q+1$ cyclic subgroups of order $q$.  If $n_p=q^2$, then $|c(G)|=q^2+q+2$. Assume that  $n_p=q$. The group $G$ has $q(p-1)$ elements of order $p$. Also, the group $G$ has unique Sylow $q$-subgroup. Thus remaining non-trivial elements of $G$ are of order $pq$.  Hence $G$ has  $q$ cyclic subgroups of order $pq$. If unique Sylow $q$-subgroup is cyclic and $n_p=q$. Then $|c(G)|=2q+3$. And, if unique Sylow $q$-subgroup is elementary abelian and $n_p=q$, then as in above case $|c(G)|=3q+2$.
\end{proof}

\begin{pp}
Let $G$ be a finite   group of order $p^2q$, where $p$ and $q$ are distinct primes and $p<q$. Then $|c(G)|=6, 2p+4, pq+4, q+4$ or $2q+2$.
\end{pp}

\begin{proof}
If $G$ is cyclic, then $|c(G)|=6$. Assume that $G$ is abelian but not cyclic. Then the unique Sylow $p$-subgroup of $G$ is elementary abelian. Thus $G$ has $p+1$ cyclic subgroups of order $p$, a unique Sylow $q$-subgroup and hence $p+1$ cyclic subgroups of order $pq$.  Thus $|c(G)|=2p+4$.
Observe that if $p\nmid q-1$, then  $G$ is abelian by Sylow's theorem. In that case, we have already computed the number of cyclic subgroups.   If $p\mid q-1$ but $p^2\nmid q-1$, then there are two non-abelian groups by \cite[Page 76-80]{bur}. The first group
 $$G_1=\langle a, b, c\mid a^q=b^p=c^p=1, bab^{-1}=a^i, ca=ac, cb=bc, \mbox{ord}_q(i)=p\rangle.$$

Observe that $G_1=\langle a, b\rangle \oplus \langle c\rangle.$ It follows from Lemma 3.1 that $\langle a, b \rangle$  has $q$ subgroups of order $p$. Let $a_1=(x,y)\in G_1$ be any element. Then  $|a_1|=p$ only if, $|x|=p$ and $|y|=p$ or $|x|=p$ and $|y|=1$ or $|x|=1$ and $|y|=p$. Thus $G_1$ has $p^2q-pq+p-1$ elements of order $p$ and hence $pq+1$ cyclic subgroups of order $p$. Similarly, it is easy to see that $G_1$ has a unique cyclic subgroups of order $q$ and $pq$. Thus $|c(G_1)|=pq+4$.

The second group
$$G_2= Z_q\rtimes Z_{p^2}= \langle a, b \mid a^q=b^{p^2}=1, bab^{-1}=a^i, \mbox{ord}_q(i)=p \rangle .$$

  Using the fact $ba= a^ib$ and $\mbox{ord}_q(i)=p$ and induction, we have 
  $b^pa=a^{i^{p}}b^p=a^{1+kq}b^p=ab^p$, for some integer $k$. 
    Thus $b{^{p}}\in Z(G_2)$. Observe that every element of $G_2$ can be written in the form $a^rb^s$, where $1\le r\le q$ and $1\le s\le p^2$. Claim that
    $$(a^rb^s)^n=a^{(r(1+i^s+i^{2s}+\ldots+ i^{(n-1)s})}b^{ns}.$$
    We prove it, by using induction on $n$. If $n=1$, then the result is trivially true. Assume that result is true for $n-1$. Now, consider  
    $$ 
\begin{array}{ccl}       
(a^rb^s)^n & = &(a^rb^s)^{n-1}(a^rb^s)\\
& = & a^{(r(1+i^s+i^{2s}+\ldots+ i^{(n-2)s})}b^{(n-1)s}a^rb^s\\
& = &a^{(r(1+i^s+i^{2s}+\ldots+ i^{(n-2)s})}a^{r(i^{(n-1)s})}b^{ns}\\
& = & a^{(r(1+i^s+i^{2s}+\ldots+ i^{(n-1)s})}b^{ns}.\\
\end{array}
 $$  
 Since $\mbox{ord}_q(i)=p$,
    $$(a^rb^s)^p=a^{(r(1+i^s+i^{2s}+\ldots+ i^{(p-1)s})}b^{ps}=1,$$
 only if, $s$ is a multiple of $p$ and $r=q$. Thus $G_2$ has $(p-1)$ elements of order $p$. The group $G_2$ has $q$ distinct  Sylow $p$-subgroups of order $p^2$. Therefore  $G_2$ has $pq(p-1)$ elements of order $p^2$. The number of elements of order $pq$ in $G_2$ are  
$$p^2q-(p-1)-pq(p-1)-q-1+1=(p-1)(q-1).$$ 
Thus $G_2$ has unique cyclic subgroup of order $pq$. Hence $|c(G_2)|=q+4$.
  
  If $p^2\mid q-1$, then it follows from \cite[page 76-80]{bur} that there are three non-abelian groups. We have both groups $G_1$ and $G_2$ of above case together with $G_3$. If $G$ is one of $G_1$ or $G_2$, then we have already calculated number of cyclic subgroups. Let
    $$G=G_3= Z_q\rtimes Z_{p^2}=\langle a, b \mid a^q=b^{p^2}=1, bab^{-1}=a^i, \mbox{ord}_q(i)=p^2\rangle.$$
Observe that $G_3$ has a unique Sylow $q$-subgroup and $q$ Sylow $p$-subgroups, respectively. Thus $G_3$ has  $pq(p-1)$ elements of order $p^2$. Since $ba=a^ib$, every element of $G_3$ will be written  $a^mb^n$ for some integer $m$ and $n$. Consider any non-identity $a^mb^n$ element of $G_3$, where $1\le m\le q$ and $1\le n< p^2$. Then as in the above case $(a^mb^n)^p=a^{rm}b^{pn}$, where $r=1+i^n+i^{2n}+\ldots +i^{(p-1)n}$. Observe that if $p\mid n$, then 
$$r=\frac{i^{pn}-1}{i^n-1}\equiv 0\; (\mbox{mod}\; q) , \mbox{because},\; \mbox{ord}_q(i)=p^2.$$ 
Thus $a^{rm}b^{pn}=1,$ only if $p|n$ and $m$ can take any value. Thus $G_3$ has $(p-1)q$ elements of order $p$ and thus $G_3$ has no subgroup of order $pq$. Hence $|c(G_3)|=2q+2.$
\end{proof}

\begin{rem}
There is one more non abelian group in the case if $p=2$  and $q=3$ of order $12$. Up to
isomorphism, there are only three non-abelian groups of order $12$: $D_{12}$, $A_4$ and $Z_3\rtimes Z_4$.
The groups $D_{12}$ (the group $G_1$) and $Z_3 \rtimes Z_4$ (the group $G_2$) have already been covered
in above theorem and $|c(A_4)|=8.$
\end{rem}

One of the  most important theorems  of group theory is Lagrange's theorem, which states that: In a finite group $G$ the order of a subgroup divides the order of the group. The converse of this theorem is true for all finite abelian groups. But, for the non abelian groups, the converse is not true in general. The smallest order of a group $G$ for which converse is not true is $12=2^23$ and the group $G$ is $A_4$. It has been known since 1799 that $A_4$ has no subgroup of order $6$. Quite surprisingly, in the following theorem, we prove that $A_4$ is the only group of the order $p^2q$, where $p<q$, which has no subgroup of the order $pq=2\times 3=6$.

\begin{thm}
Let $G$ be a finite non-abelian group of the order $p^2q$, where $p$ and $q$ are
primes and $p<q$, such that the converse of Lagrange's theorem is not true. Then $G = A_4$
\end{thm}

\begin{proof}
It is sufficient to show that if $G\neq A_4$, then $G$ has a subgroup of order $pq$.  Let $|G|=p^2q$, where $p<q$. Then it follows from the proof of the Proposition 3.3, $G$ has either a normal subgroup of order $p$  or $q$. Thus except $A_4$, every group $G$ of order $p^2q$ has a subgroup of order $pq$. This proves the result.

\end{proof}

\begin{pp}
Let $G$ be a finite abelian $p$-group such that $G= C_{p^n}\times C_{p^m}$, where $n\ge m$. Then number of cyclic subgroups of order $p^r$ is
$$
=
\left\{
\begin{array}{cc}
p^{r-1}(p+1) & \mbox{if}\;\; 1\le r \le m\\
p^{m} & \mbox{if}\;\; m<r\le n

\end{array}
\right\}. 
$$
\end{pp}

\begin{proof}
First assume that $1\le r\le m$. Let $a^ib^j$, where $1\le i\le p^n$ and $1\le j\le p^m$ be any element of $G$. Observe that $(a^ib^j)^{p^r}=1$ only if $i$ is a multiple of $p^{n-r}$ and $j$ is a multiple of $p^{m-r}$. Thus each $i$ and $j$ has $p^r$ choices. Hence the number of elements $a^ib^j$ such that  $(a^ib^j)^{p^r}=1$ are $p^{2r}$. Similarly  number of elements $a^ib^j$ such that  $(a^ib^j)^{p^{r-1}}=1$ are $p^{2r-2}$. Thus $G$ has $p^{2r}-p^{2r-2}$ elements of order $p^r$, and number of cyclic subgroups of order $p^r$ is $\frac{p^{2r}-p^{2r-2}}{p^{r-1}(p-1)}=p^{r-1}(p+1).$ If $m< r \le n$, then as above, it is not very hard to show that the number of elements $a^ib^j$ such that  $(a^ib^j)^{p^r}=1$ is $p^{r+m}$ and number of elements $a^ib^j$ such that  $(a^ib^j)^{p^{r-1}}=1$ are $p^{r-1+m}$ respectively. Thus $G$ has $p^m$ cyclic subgroups of order $p^r$.
\end{proof}

\begin{rem}
The above result can be generalized for arbitrary finite abelian $p$-group. Using the same technique, it can be seen that if $G=C_{p^2}\times C_p\times C_{p},$ then $G$ has $p^2+p+1$ cyclic subgroups of order $p$ and $p^2$ cyclic subgroups of order $p^2$. Thus $|c(G)|=2p^2+p+2$. 
\end{rem}

The following table represent the number of cyclic subgroups for all the finite groups of the order $p^3$. If $G$ is an abelian group, then $|c(G)|$ is  obtained from Proposition 3.6 and, if $G$ is non-abelian $p$-group, where $p$ is odd, then the result follows from  \cite[Page  62 and 64]{sta}. It is also worthwhile to mention here, statha et al. in \cite{sta} used combinotorial technique to compute the number of cyclic subgroups of a finite group. But in this paper, we used only group theoretic techniques only.

\begin{table}[H]
\centering
\caption{} \label{table1}

\begin{tabular}{|c|c|c|}
 \hline 
 S. No. & Groups of order $p^3$ & $|c(G)|$ \\ 
 \hline 
 1. & $G=C_{p^3}$ & $4$ \\ 
 \hline 
 2. & $G=C_{p^2}\times C_{p}$ & $2p+2$ \\ 
 \hline 
 3. & $G=C_{p}\times C_p\times C_p$ & $p^2+p+2$  \\ 
 \hline 
 4. & $G=D_8$ & $7$ \\ 
 \hline 
 5. & $G=Q_8$ & 5 \\ 
 \hline 
 6. & $G=\langle a, b, c \mid a^p=b^p=c^p=1, ba=ab, ca=abc, cb=bc\rangle, p>2$ & $p^2+p+2$ \\ 
 \hline 
 7. & $G=\langle a, b \mid a^{p{^2}}=b^p=1, ba=a^{p+1}b\rangle, p>2$ & $2p+2$ \\ 
 \hline 
 \end{tabular}  
 
\end{table}

Next, we calculate the number of cyclic subgoups of all finite $p$-groups of order $p^4$. The list of groups of order $2^4$ is given in \cite{sag}. There are 14 groups of order 16. Out of them nine are non-abelian. We write the $i$th group of order 16 in the list as $G_{i}$ and generators $1,2,3$ and $4$ respectively $x$, $y$, $z$ and $w$. The groups $G_6$ to $G_{11}$ are of nilpotency 2 and $G_{12}$ to $G_{14}$ are of class 3. In the following proposition, we take one group $G_6$ out of the groups $G_6$ to $G_{11}$ of class 2 and a group $G_{12}$ out of the groups $G_{12}$ to $G_{14}$ of class 3 and  find the number of cyclic subgroups for these groups. Similarly, one can  compute $|c(G)|$ if $G$ is one of $G_7$ to  $G_{11}$ or $G$ is one of $G_{12}$ or $G_{13}$. We write $|c(G_i)|$ for $6\le i\le 14$ in Table 2.

\begin{pp}
Let $G$ be a non-abelian group of order $2^4$. If $$G=G_6= \lbrace x^4,y^2, z^2, [x,y]x^2, [x,z], [y,z]\rbrace,$$ then $|c(G)|= 14$ and if 
$$G=G_{12}=\lbrace x^8, y^2, [x,y]x^2\rbrace,$$
then $|c(G)|=12$.
\end{pp}

\begin{proof}
First suppose that $G=G_6$.  Observe that every element of $G$ will be written in the form $x^iy^jz^k$, where $1\le i \le 4, 1\le j\le 2$ and $1\le k\le 2$ and $Z(G)=\langle x^2, z \rangle$.  Thus nilpotency class of $G$ is 2, and hence $(x^iy^jz^k)^2=x^{2i}y^{2j}[y^j, x^i]=x^{2i}y^{2j}x^{2ij}=1$, only if $j=1$ and $i, k $ can take any values or $j=0$ and $i=2,4$ and $k$ can take any value. Also $(x^iy^jz^k)^4=1$ for any values of $i, j$ and $k$. Thus $G$ has 11 elements of order 2 and $4$ elements of order 4 and one identity. Hence $|c(G)|=14$. Let 
$$G=G_{12}=\lbrace x^8, y^2, [x,y]x^2\rbrace.$$ Observe that $cl(G)=3$, $yx=x^7y$ and every element of $G$ will be written in the form $x^iy^j$, where $1\le i\le 8$ and $1\le j\le 2$ and 
$$
(x^iy^j)^2=
\left\{
\begin{array}{cc}
x^{2i} & \mbox{if}\; j=0\\
x^{8i} & \mbox{if}\; j=1

\end{array}
\right\}.
$$
Thus $G$ has 9 elements of order 2, two elements of  order 4 and 4 elements of order 8. Hence $|c(G)|=12$.
 \end{proof}

From the Propositions 3.6 and 3.8 and Remark 3.7, we constitute the following table (Table 2) consisting the number of cyclic subgroups of a group $G$ for the cases, either $G$ is an abelian group of order $p^4$ or $G$ is a non-abelian group of order 16.

\begin{table}[H]
\centering
\caption{} \label{table2}
\begin{tabular}{|c|c|c|}

\hline 
S. No. &   Abelian groups of order $p^4$ and non-abelian groups of order $2^4$  & $|c(G)|$\\ 
\hline 
$1.$ & $G=C_{p^4}$ & $5$ \\ 
\hline 
$2.$ & $G=C_{p^3}\times C_{p}$ & $3p+2$ \\ 
\hline 
$3.$ & $G=C_{p^2}\times C_{p^2}$ & $p^2+2p+2$ \\ 
\hline 
$4.$ & $G=C_{p^2}\times C_{p}\times C_{p}$ & $2p^2+p+2$ \\ 
\hline 
$5.$ & $G=C_p\times C_p\times C_p\times C_p$ & $p^3+p^2+p+2$ \\ 

\hline 
$6.$ & $G=G_6=\lbrace  x^4, y^2, z^2, [x,y]x^2, [x,z], [y,z]\rbrace$ & $14$ \\ 
\hline 
$7.$ & $G=G_7=\lbrace x^2y^{-2}, z^2, [x,y]x^2, [x,z], [y,z]\rbrace$ & $10$ \\ 
\hline 
$8.$ & $G=G_8=\lbrace  y^2, z^2, [y,z]x^2, [x,y], [x,z]  \rbrace$ & $12$ \\ 
\hline 
$9.$ & $G=G_9= \lbrace x^2, y^4, [x,y,x], [x,y,y]\rbrace$ & $12$ \\ 
\hline 
$10.$ & $G=G_{10}= \lbrace x^4, y^4, [x,y]x^2 \rbrace$ & $10$ \\ 
\hline 
$11.$ & $G=G_{11}= \lbrace x^2y^{-8}, [x,y]y^4\rbrace$ & $8$ \\ 
\hline 
$12.$ & $G=G_{12}=\lbrace x^8, y^2, [x,y]x^2\rbrace$ & $12$ \\ 
\hline 
$13.$ & $G=G_{13}=\lbrace y^2, [y,x]x^2\rbrace$ & $10$ \\ 
\hline 
$14.$ &$G=G_{14}=\lbrace x^4y^{-2}, [x,y]x^2\rbrace$ & $8$ \\ 
\hline 
\end{tabular}
\end{table}

Here, we assume that $G$ is a non-abelian group of order $p^4$, where $p$ an odd prime. It follows from \cite{bur} that there are 10 such groups upto isomorphism. These groups are also written on \cite[Page 35]{sta}. In \cite{sta} these groups are numbered from (vi) to (xv). In this paper, we number them $G\mbox{(vi)}$ to $G\mbox{(xv)}$. Observe that the groups $G\mbox{(xii)}$, $G\mbox{(xiii)}$ and $G\mbox{(xv)}$ have two different presentations, one for $p=3$ and the other for $p>3$. In \cite[Chapter 2]{sta} the number of cyclic subgroups of these groups have been calculated except for the groups $G\mbox{(xii)}$, $G\mbox{(xiii)}$ and $G\mbox{(xv)}$ and  when $p=3$.   In the following proposition, we calculate $|c(G)|$, for these groups $G$ when $p=3$. 

\begin{pp}
Let $G$ be a non-abelian group of order $3^4$ such that $G$ is one of the $G\mbox{(xii)}$, $G\mbox{(xiii)}$ or $G\mbox{(xv)}$. Then $|c(G)|=17$, $23$ or $35$.
\end{pp}

\begin{proof}
First, suppose that
 $$G=G\mbox{(xii)}=\langle a, b, c \mid a^9=b^3=1, c^3=a^3, ab=ba^4, ac=cab^2,bc=cb\rangle.$$ 
 Consider, $[a^3,b]=[a,b]^3=a^9=1$, thus $a^3\in Z(G)$. Using this fact $ca=abc$ and $ba=a^7b$, we have $c^ka^i=a^{i+3ki(i-1)}b^{ik}c^k$ and $b^ja^i=a^{i(1+6j)}b^j$, where $1\le i\le 9$, $1\le j\le 3$ and $1\le k\le 3$. Thus every element of $G$ is written in the form $a^ib^jc^k$ for some integers $i,j$ and $k$. From the above relations, we get $(a^ib^jc^k)^3=a^{3(i+k+2i^2k)}$. Therefore $(a^ib^jc^k)^3=1,$ only if $1\le j\le 3$, $k=3$ and $i\in \lbrace 3, 6, 9\rbrace.$ Thus $G$ has 8 elements of order $3$. Also, observe that $\mbox{exp}(G)=9$. Thus $G$ has $72$
 elements order $9$. Hence $|c(G)|=17$. For other two groups, similarly $|c(G)|$ can be calculated.
\end{proof}

\begin{table}[H]
\centering
\caption{} \label{table3}

\begin{tabular}{|c|c|c|}
\hline 
S. No. & Non-abelian groups of order $p^4$,  $p$ odd & $|c(G)|$ \\ 
\hline 
$1.$ & $G{(\mbox{vi})}=\langle a, b\mid  a^{p^3}=b^p=1, ba=a^{1+p^2}b \rangle$ & $3p+2$ \\ 
\hline 
$2.$ & $G(\mbox{vii})= \langle a, b, c \mid a^{p^2}=b^p=c^p=1, ba=ab, ca=ac, cb=a^pbc\rangle$ & $2p^2+p+2$ \\ 
\hline 
$3.$ & $G(\mbox{viii})=\langle a, b \mid a^{p^2}=b^{p^2}=1, ba=a^{1+p}b\rangle$ & $p^2+2p+2$ \\ 
\hline 
$4.$ & $G(\mbox{ix})= \langle a, b, c \mid a^{p^2}=b^p=c^p=1, ba=a^{1+p}b, ca=ac, cb=bc\rangle$ & $2p^2+p+2$ \\ 
\hline 
$5.$ & $G(\mbox{x})=\langle a, b, c \mid a^p=b^p=c^p=1, ba=ab, ca=abc, bc=cb\rangle$ & $2p^2+p+2$ \\ 
\hline 
$6.$ & $G(\mbox{xi})=\langle a, b, c \mid a^{p^2}=b^p=c^p=1, ba=a^{1+p}b, ca=abc, bc=cb\rangle$ & $2p^2+p+2$ \\ 
\hline 
$7(i).$ & $ G(\mbox{xii})=\langle a, b, c\mid a^{p^2}=b^p=c^p=1, ba=a^{1+p}b,$\\
&$ ca=a^{1+p}bc, cb=a^pbc\rangle,\; p>3$ & $2p^2+p+2$ 
\\
\hline 
$7(ii).$ & $G(\mbox{xii})= \langle a, b, c\mid a^{p^2}=b^p=1, c^p=a^p,$\\
& $ab=ba^{1+p}, ac=cab^{-1}, cb=bc\rangle,\; p=3$ & $17$ \\ 
\hline 
$8(i).$ & $G(\mbox{xiii})=\langle a, b, c \mid a^{p^2}=b^p=c^p=1, ba=a^{1+p}b,$\\
& $ ca=a^{1+dp}bc, cb=a^{dp}bc, d\not\equiv 0, 1(\mbox{mod p})\rangle, p>3$ & $2p^2+p+2$  \\ 
\hline 
$8(ii).$& $G(\mbox{xiii})=\langle a, b, c \mid a^{p^2}=b^{p}=1, c^p=a^{-p},$\\
& $ab=ba^{1+p}, ac=cab^{-1}, cb=bc\rangle, \mbox{for}\; p=3$ & $23$ \\ 
\hline 
$9.$ & $G(\mbox{xiv})=\langle a, b, c, d \mid a^p=b^p=c^p=d^p=1,$ \\
& $dc=acd, bd=db, ad=da, bc=cb, ac=ca, ab=ba\rangle$  & $p^3+p^2+p+2$ \\
\hline
$10(i).$ & $G(\mbox{xv})=\langle a, b, c, d \mid a^p=b^p=c^p=d^p=1,$\\
& $ba=ab, ca=ac, da=ad, cb=bc, db=abd, dc=bcd\rangle, p>3$ & $p^3+p^2+p+2$  \\
\hline
$10(ii).$ & $G(\mbox{xv})=\langle a, b, c \mid a^{p^2}=b^p=c^p=1, ab=ba,$\\
& $ac=cab, bc=ca^{-p}b\rangle, p=3$ & $35$ \\
\hline
\end{tabular} 
\end{table}

It follows from Table 3 that if $G$ is a  non-abelian group of order $p^4$, ($p$-odd prime). Then $|c(G)|\ge 11$.

\section{Main Results}

We are now ready to prove our main results.

\begin{thm}
Let $G$ be a finite group. Then $|c(G)|=6$ if and only if $G\simeq C_{p^5}$, $C_{p^2q}$, $C_{pq^2}$ or $C_2\times C_4$, where $p$ and $q$ are distinct primes and $p<q$. 
\end{thm}

\begin{proof}
Assume that $|c(G)|=6$. Thus order $G$ is divisible by at most two distinct primes by Lemma 2.1.  If $|G|=p^aq^b$, where $p<q$ are distinct primes, then  either $1\le a\le 2$ and $b=1$ or $a=1$ and $1\le b\le 2$. And, also If, either $a=1$ and $b=2$ or $a=2$ and $b=1$, then $G$ is cyclic. The order of $G$ cannot  be $pq$ by Lemma 3.1. Next, assume that  $G$ is a finite $p$-group. Thus $|G|\leq p^5$. If $|G|=p^5$, then $G=C_{p^5}$ and if $|G|\le p^4$, then the result follows from Table 1, Table 2 and Table 3 that $G=C_2\times C_4$.
\end{proof}

\begin{thm}
Let $G$ be a finite group. Then $|c(G)|=7$ if and only if $G$ is one of the $D_8, D_{10}, Z_3\rtimes Z_4, C_5\times C_5$ or $C_{p^6}$, where $p$ is a prime.
\end{thm}

\begin{proof}
First suppose that $|c(G)|=7$. Let $|G|=p^aq^b$, where $p<q$, then  either $1\le a\le 2$ and $b=1$ or $a=1$ and $1\le b\le 2$. If $a=b=1$, then $G=D_{10}$ by Lemma 3.1.  There do not exist any group if $a=1$ and $b=2$ by Proposition 3.2. If $a=2$ and $b=1$, then it follows from the proof of Proposition 3.3 that $G=Z_3\rtimes Z_4$. Next, we assume that $G$ is a $p$-group. If $|G|=p^6$, then $G=C_{p^6}$. Assume that $|G|=p^5$. If $G$ has a non cyclic subgroup of order $p^4$, then it follows from Table 2 and Table 3 that there do not exist any group. Assume that every subgroup of order $p^4$ is cyclic.  Let $H$ be a cyclic subgroup of order $p^4$. Observe that $G$ has at least $p^5-p^4\ge p^4$  elements that are distinct from  the elements of $H$. Since every element of $G$ must lie in some cyclic subgroup,  $|c(G)|\ge 9$. If $p^3\le |G|\le p^4$, then it follows from Table 1, Table 2 and Table 3 that  $G=D_8.$ And, if $|G|=p^2$, then $G=C_5\times C_5$.
\end{proof}

\begin{thm}
Let $G$ be a finite group. Then $|c(G)|=8$ if and only if $G$ is isomorphic to one of the followings:
\begin{enumerate}
\item[$(a)$] $C_{pqr}, C_{p^3q}$, $C_{pq^3}$, $C_{p^7}$ or an abelian  but non-cyclic group of order $4q$, where $p<q<r,$ are distinct primes,
\item[$(b)$] $C_{2^3}\times C_2, C_2\times C_2\times C_2$ or  $C_{3^2}\times C_3$,
\item[$(c)$] $G_{11}=\lbrace x^2y^{-8}, [x,y]y^4\rbrace, G_{14}=\lbrace x^4y^{-2}, [x,y]x^2\rbrace,$  
 $\langle a, b\mid a^{3^2}=b^3=1, ba=a^4b\rangle$, or $A_4$.
\end{enumerate}
\end{thm}

\begin{proof}
Assume that $|c(G)|=8$. Then the order of $G$ is divisible by at most three distinct primes. If $|G|=p^aq^br^c$, where $a, b$ and $c$ are positive integers, then $G=C_{pqr}$. Let  $|G|=p^aq^b$. Then, either $1\le a\le 3$ and $b=1$ or $a=1$ and $1\le b\le 3$. If $|G|=p^3q$ or $pq^3$, then $G$ is cyclic. The order $G$ cannot be $pq$ by Lemma 3.1. Assume that $|G|=p^2q$. It follows from the proof of the Proposition 3.3 and Remark 3.4 that either $G$ is an abelian of order $4q$, where $q$ is odd prime, or $G=A_4$. If $a=1$ and $b=2$, then there do not exist any group by Proposition 3.2. Next, we suppose that $G$ is a $p$-group. If $|G|=p^7$, then $G=C_{p^7}$. There does not exist any group of order $p^5$ or $p^6$ can be proved using the similar arguments as in Theorem 4.2. If $|G|=p^4$, then $G$ is one of the $C_{2^3}\times C_2$, $G_{11}$ or $G_{14}$ in Table 2. If $|G|=p^3$, then by Table 1, $G$ is one of the $C_2\times C_2\times C_2$ or $C_{3^2}\times C_3$ or $G=\langle a, b\mid a^{3^2}=b^3=1, ba=a^4b\rangle.$
\end{proof}

In the next result, we gave a very short proof of the main theorem  of Zhou \cite[Theorem 2.4]{zho}.
\begin{cor}
Let $G$ be a finite group. Then $|c(G)|\le 5$ if and only if $G$ is a subgroup of $C_{p^4}$ or $G\simeq  S_3, C_3\times C_3, Q_8, C_{pq}$ or $C_2\times C_2$, where $p$ and $q$ are distinct primes.
\end{cor}

\begin{proof}
We first assume that $|c(G)|=1$ or $2$. Then trivially $|G|=1$ or $p$. Since $|c(G)|\le 5$, either $G$ is a $p$-group of order less that or equal to $p^4$ or $|G|=pq$. Thus, the result follows from Lemma 2.1, Lemma 3.1, Table 1, Table 2 and Table 3.
\end{proof}

\noindent{Acknowledgment}: The research is supported by Thapar University under the seed money project no: TU/DORSP/57/473 and gratefully acknowledged.

\

\end{document}